\documentclass{amsart}
\usepackage{amssymb}
\usepackage{amsmath}
\usepackage{latexsym}
\usepackage{eepic}
\usepackage{epsfig}
\usepackage{graphicx}
\usepackage{pb-diagram,lamsarrow,pb-lams,amscd}
\usepackage{eufrak}
\usepackage{mathrsfs}
\usepackage[english]{babel}
\usepackage[all,cmtip]{xy}

\DeclareMathAlphabet{\mathpzc}{OT1}{pzc}{m}{it}
\newtheorem{theorem}{Theorem}[subsection]
\newtheorem{theorem-definition}[theorem]{Theorem-Definition}
\newtheorem{lemma-definition}[theorem]{Lemma-Definition}
\newtheorem{definition-prop}[theorem]{Proposition-Definition}

\newtheorem{prop}[theorem]{Proposition}

\newtheorem{cor}[theorem]{Corollary}
\newtheorem{definition}[theorem]{Definition}

\newtheorem{conjecture}[theorem]{Conjecture}

\theoremstyle{definition}
\newtheorem{remark}[theorem]{Remark}
\newtheorem{example}[theorem]{Example}

\newtheorem{question}[theorem]{Question}

\newcommand{\LL}{\ensuremath{\mathbb{L}}}

\newcommand{\N}{\ensuremath{\mathbb{N}}}
\newcommand{\Z}{\ensuremath{\mathbb{Z}}}
\newcommand{\Q}{\ensuremath{\mathbb{Q}}}

\newcommand{\R}{\ensuremath{\mathbb{R}}}
\newcommand{\C}{\ensuremath{\mathbb{C}}}

\newcommand{\A}{\ensuremath{\mathbb{A}}}

\newcommand{\X}{\ensuremath{\mathscr{X}}}

\newcommand{\cX}{\ensuremath{\mathcal{X}}}

\renewcommand{\R}{\ensuremath{\mathbb{R}}}
\renewcommand{\C}{\ensuremath{\mathbb{C}}}

\renewcommand{\A}{\ensuremath{\mathbb{A}}}
\renewcommand{\X}{\ensuremath{\mathfrak{X}}}

\newcommand{\Spec}{\ensuremath{\mathrm{Spec}\,}}
\newcommand{\Spf}{\ensuremath{\mathrm{Spf}\,}}
\newcommand{\Sp}{\ensuremath{\mathrm{Sp}\,}}
\newcommand{\Proj}{\ensuremath{\mathbb{P}}}

\newcommand{\an}{\mathrm{an}}
\newcommand{\Hom}{\mathrm{Hom}}
\newcommand{\sh}{sh}

\newcommand{\redu}{\mathrm{red}}
\newcommand{\rig}{\mathrm{rig}}

\newcommand{\Var}{\mathrm{Var}}
\newcommand{\Sm}{\mathrm{Sm}}

\newcommand{\KVar}[1]{\ensuremath{K_0(\Var_{#1})}}
\newcommand{\Kmod}[1]{\ensuremath{K_0^{\mathrm{mod}}(\Var_{#1})}}

\newcommand{\Lie}{\mathrm{Lie}}
\newcommand{\length}{\mathrm{length}}
\newcommand{\coker}{\mathrm{coker}}
\newcommand{\ord}{\mathrm{ord}}

\numberwithin{equation}{section} \hyphenpenalty=6000
\begin{document}
\title[Motivic zeta functions]{Motivic zeta functions for degenerations of abelian varieties and Calabi-Yau varieties}
\author{Lars Halvard Halle}
\address{Matematisk Institutt\\
Universitetet i Oslo\\Postboks 1053
\\
Blindern\\0316 Oslo\\ Norway} \email{larshhal@math.uio.no}
\author[Johannes Nicaise]{Johannes Nicaise}
\address{KULeuven\\
Department of Mathematics\\ Celestijnenlaan 200B\\3001 Heverlee \\
Belgium} \email{johannes.nicaise@wis.kuleuven.be}

\thanks{
 The second author was partially supported by the Fund for Scientific Research - Flanders (G.0415.10). 
} \maketitle

\section{Introduction}
Let $f \in \mathbb{Z}[x_1, \ldots, x_n]$ be a non-constant
polynomial, and let $p$ be a prime. Igusa's $p$-adic zeta function
$Z^p_f(s)$ is a meromorphic function on the complex plane that
encodes the number of solutions of the congruence $f\equiv 0$
modulo powers $p^m$ of the prime $p$. Igusa's {\em $p$-adic
monodromy conjecture} predicts in a precise way how the
singularities of the complex hypersurface defined by the equation
$f=0$ influence the poles of $Z^p_f(s)$ and thus the asymptotic
behaviour of this number of solutions as $m$ tends to infinity.
The conjecture states that, when $p$ is sufficiently large, poles
of $Z^p_f(s)$ should correspond to local monodromy eigenvalues of
the polynomial map $\C^n\to \C$ defined by $f$. We refer to
Section \ref{sec-zetafunctions} for a precise formulation.

Starting in the mid-nineties,  J. Denef and F. Loeser developed
the theory of motivic integration, which had been introduced by M.
Kontsevich in his famous lecture at Orsay in 1995. Denef and
Loeser used this theory to construct a motivic object
$Z^{mot}_f(s)$ that interpolates the $p$-adic zeta functions
$Z^p_f(s)$ for $p\gg 0$ and captures their geometric essence. This
object is called the {\em motivic zeta function} of $f$. Denef and
Loeser also formulated a motivic upgrade of the monodromy
conjecture (Conjecture \ref{conj-motmon}). Its precise relation
with the $p$-adic monodromy conjecture is explained in Section
\ref{subsec-comparmon}.

The aim of this paper is to present  a \emph{global} version of
Denef and Loeser's motivic zeta functions. Let $X$ be a Calabi-Yau
variety over a complete discretely valued field $K$ (i.e., a
smooth, proper and geometrically connected variety with trivial
canonical sheaf). We'll define the motivic zeta function $Z_X(T)$
of $X$. This is a formal power series with coeffients in a certain
localized Grothendieck ring of varieties over the residue field
$k$ of $K$. We'll show that $Z_X(T)$ has properties analogous to
Denef and Loeser's zeta function, and we'll prove a global version
of the motivic monodromy conjecture
 when $X$ is an abelian variety, under a certain tameness
condition on $X$ (Theorem \ref{thm-main}).

The link between Denef and Loeser's motivic zeta function
$Z^{mot}_f(s)$ and our global variant is an alternative
interpretation of $Z^{mot}_f(s)$ in terms of non-archimedean
geometry, due to J. Sebag and the second author \cite{NiSe-Inv}.
 This interpretation is based on the theory of motivic integration on rigid varieties developed by F. Loeser
and J. Sebag \cite{motrigid}, which explains how one can associate
a {\em motivic volume} to a gauge form on a smooth rigid variety
over a complete discretely valued field. J. Sebag and the second
author constructed the {\em analytic Milnor fiber} of a
hypersurface singularity, a non-archimedean model for the
classical Milnor fibration in the complex analytic setting. The
analytic Milnor fiber is a smooth rigid variety over a field $K$
of Laurent series. The motivic zeta function can be realized as a
generating series whose coefficients are motivic volumes of a
so-called {\em Gelfand-Leray} form on the analytic Milnor fiber
over finite totally ramified extension of the base field $K$. This
is explained in detail in Section \ref{sec-motrig}.

This interpretation of the motivic zeta function admits a natural
generalization to the global case, where we replace the analytic
Milnor fiber by a Calabi-Yau variety $X$ over a complete
discretely valued field $K$ and the Gelfand-Leray form by a
suitably normalized gauge form $\omega$ on $X$. The zeta function
$Z_X(T)$ is studied in Section \ref{sec-abelian} when $X$ is an
abelian variety and in Section \ref{sec-degcy} in the general
case. We raise the question whether there exists a relation
between the poles of $Z_X(T)$ and the monodromy eigenvalues of $X$
as predicted by the monodromy conjecture in the case of
hypersurface singularities (Question \ref{ques-GMP}).

We studied motivic zeta functions of abelian varieties  in detail
in the papers \cite{HaNi-comp,HaNi-zeta,HaNi-jumpmon}. Section
\ref{sec-abelian} gives an overview of the results and methods
used in those papers.  A powerful and central tool is the N\'eron
model of an abelian $K$-variety $A$, which is the ``minimal''
extension of $A$ to a smooth group scheme over $R$. The N\'eron
model $\mathcal{A}$ of $A$ comes equipped with much interesting
structure, such as the Chevalley decomposition of the identity
component of its special fiber, the Lie algebra $\Lie(\mathcal{A})
$ and the component group $\Phi_A$. The key point in the study of
$Z_A(T)$ is to understand how these objects change under ramified
extensions of $K$.

Our main result is Theorem  \ref{thm-main}, which states that if
$A$ is a tamely ramified abelian $K$-variety, then $
Z_A(\mathbb{L}^{-s}) $ is rational with a unique pole at $s=c(A)$,
where $c(A)$ denotes Chai's base change conductor of $A$
\cite{chai}. Moreover, for every embedding of $\Q_\ell$ in $\C$,
the complex number $\exp(2 \pi c(A) i)$ is an eigenvalue of the
monodromy transformation on $ H^g(A \times_K K^t,
\mathbb{Q}_{\ell})$, where $K^t$ denotes a tame closure of $K$,
and where $g$ is the dimension of $A$. This shows that a global
version of Denef and Loeser's motivic monodromy conjecture holds
for tamely ramified abelian varieties.

The situation for general Calabi-Yau varieties is at the moment
far less clear than in the abelian case. Our proofs for abelian
varieties rely heavily on the theory of N\'eron models, and these
methods do not extend to the general case. However, if we restrict
ourselves to equal characteristic zero, there is still much that
can be said, and we present some of our results under this
assumption in Section \ref{sec-degcy}.

A particular advantage in characteristic zero,  and the basis for
many applications, is that we can find an $sncd$-model of $X$,
i.e., a regular proper $R$-model $\mathcal{X}$ whose special fiber
$ \cX_s$ is a divisor with strict normal crossings. We explain in
Section \ref{sec-degcy} how the results in \cite{NiSe-Inv} yield
an explicit expression for $Z_X(T)$ in terms of the model
$\mathcal{X}$. This expression shows that $Z_X(T)$ is rational,
 and yields a finite subset of $\Q$ that contains all the poles of
 $Z_X(\LL^{-s})$.
However, due to cancellations in the formula, it is often
difficult to use this description to  determine the precise set of
 poles.

The opposite of the largest pole of $Z_X(\mathbb{L}^{-s})$ turns
out to be an interesting  invariant of $X$, we call it the log
canonical threshold $lct(X)$ of $X$. It can be easily computed on
the model $\cX$. We can show that $lct(X)$ corresponds to a
monodromy eigenvalue on the degree $\mathrm{dim}(X)$ cohomology of
$X$. The value $lct(X)$ is a global version of the log canonical
threshold for complex hypersurface singularities, we explain the
precise relationship in Section \ref{subsec-lctcomparison}. Since
we know that for an abelian $K$-variety $A$, the base change
conductor $c(A)$ is the unique pole of $Z_A(\LL^{-s})$, we find
that $lct(A)=-c(A)$. This yields an interesting relation between
the N\'eron model of $A$ and the birational geometry of
$sncd$-models of $A$. Our explicit expression for the zeta
function allows to compute many other arithmetic invariants of $A$
on an $sncd$-model, in particular the number of connected
components of the N\'eron model. This generalizes the results that
were known for elliptic curves. Conversely, we can use the zeta
function to extend many interesting invariants of abelian
$K$-varieties to arbitrary Calabi-Yau varieties.

\section{Preliminaries}\label{sec-preliminaries}
\subsection{Notation}
For every ring $A$, an $A$-variety is a reduced separated
$A$-scheme of finite type. An algebraic group over a field $F$ is
a reduced group scheme of finite type over $F$. We denote by $\mu$
the profinite group scheme of roots of unity.
\subsection{Local monodromy eigenvalues}\label{subsec-locmon}
Let $k$ be a subfield of $\C$, let $X$ be a $k$-variety, and let
$$f:X\to \A^1_k=\Spec k[t]$$ be a $k$-morphism. Let $x$ be a point
of $X(\C)$ such that $f(x)=0$. We denote by $X^{\an}$ the complex
analytification of $X\times_k \C$, by $f^{\an}:X^{\an}\to \C$ the
complex analytic map induced by $f$, and by $X_s^{\an}$ the zero
locus of $f^{\an}$ in $X$. We say that a complex number $\alpha$
is a local monodromy eigenvalue of $f$ at $x$ if there exists an
integer $j\geq 0$ such that $\alpha$ is an eigenvalue of the
monodromy transformation on $R^j\psi_{f^{\an}}(\C)_x$. Here
$$R\psi_{f^{\an}}(\C)\in D^b_c(X^{\an}_s,\C)$$ denotes the
 complex of nearby cycles associated to $f^{\an}$ \cite[\S4.2]{dimca}. If $X$ is smooth at $x$, then the complex vector space
$R^j\psi_{f^{\an}}(\C)_x$ is isomorphic to the degree $j$ singular
cohomology space of the Milnor fiber of $f^{\an}$ at the point
$x$.

If $f$ is a polynomial in $k[x_1,\ldots,x_n]$, then we can speak
of local monodromy eigenvalues of $f$ by considering $f$ as a
morphism $\A^n_k\to \A^1_k$.
\subsection{The Bernstein-Sato polynomial}\label{subsec-bernstein}
Let $k$ be a field of characteristic zero. Let $X$ be a smooth
irreducible $k$-variety of dimension $n$, endowed with a morphism
$$f:X\to \A^1_k=\Spec k[t].$$ Denote by $X_s$ the fiber of $f$
over the origin. For every closed point $x$ of $X_s$, we denote by
$k_x$ the residue field at $x$ and by $b_{f,x}(s)$ the
Bernstein-Sato polynomial of the formal germ of $f$ in
$\widehat{\mathcal{O}}_{X,x}\cong k_x[[x_1,\ldots,x_n]]$ (see
\cite[3.3.6]{bjork}). We call $b_{f,x}(s)$ the local
Bernstein-Sato polynomial of $f$ at $x$. If $k=\C$, then
$b_{f,x}(s)$ coincides with the Bernstein polynomial of the
analytic germ of $f$ in $\mathcal{O}_{X^{\an},x}$, by
\cite[\S4.2]{mebkhout}.

If $h:Y\to X$ is a morphism of smooth $k$-varieties and $y$ is a
closed point of $Y$ such that $x=h(y)$ and $h$ is \'etale at $x$,
then the faithfully flat local homomorphism
$\widehat{\mathcal{O}}_{X,x}\to \widehat{\mathcal{O}}_{Y,y}$
satisfies the conditions in \cite[\S4.2]{mebkhout}. It follows
 that $b_{f,x}(s)=b_{f\circ h,y}(s)$. The
same argument shows that $b_{f,x}(s)$ is invariant under arbitrary
extensions of the base field $k$. If $k=\C$, then Kashiwara has
shown that the roots of $b_{f,x}(s)$ are rational numbers
\cite{kashiwara}. By \cite{saito-SMF}, they lie in the interval
$]-n,0[$. Invoking the Lefschetz principle, we see that these
properties hold for arbitrary $k$.

If  $k=\C$, then it was proven by Malgrange \cite{malgrange-ast}
that, for every root $\alpha$ of the local Bernstein-Sato
polynomial $b_{f,x}(s)$, the value $\exp(2\pi i \alpha)$ is a
local monodromy eigenvalue of $f$ at some point of  $X^{\an}_s$.
Moreover, if we allow $x$ to vary in the zero locus of $f$, all
local monodromy eigenvalues arise in this way. Since $b_{f,x}(s)$
is invariant under extension of the base field $k$, this property
still holds over all subfields $k$ of $\C$.

By constructibility of the nearby cycles complex, the local
monodromy eigenvalues of $f$ form a finite set $\mathrm{Eig}(f)$.
 Thus, as $x$ runs through the set of closed points of $X_s$, the
polynomials $b_{f,x}(s)$ form a finite set, since they are all
monic polynomials whose roots belong to the finite set of rational
numbers $\alpha$ in $[-n,0[$ such that $\exp(2\pi i \alpha)$ lies
in $\mathrm{Eig}(f)$.
 We call the least
common multiple of the polynomials $b_{f,x}(s)$ the Bernstein-Sato
polynomial of $f$, and we denote it by $b_{f}(s)$. If $X=\A^n_k$,
then by \cite[\S4.2]{mebkhout}, this definition coincides with the
usual definition of the Bernstein-Sato polynomial of an element
$f$ in $k[x_1,\ldots,x_n]$.

\section{$P$-adic and motivic zeta functions}\label{sec-zetafunctions}
\subsection{The Poincar\'e series}
Let $f$ be an element of $\Z[x_1,\ldots,x_n]\setminus \Z$, for
some integer $n>0$, and let $p$ be a prime number. For every
integer $m\geq 0$, we denote by $S_m$ the set of solutions of the
congruence $f\equiv 0$ modulo $p^{m+1}$, i.e.,
$$S_m= \{a\in (\Z/p^{m+1}\Z)^n\,|\,f(a)\equiv 0 \mod
p^{m+1}\}.$$ We put $N_m=\sharp S_m$.

\begin{definition}
The Poincar\'e series associated to $f$ and $p$ is the generating
series
$$P(T)=\sum_{m\geq 0}N_m T^m\quad \in \Z[[T]].$$
\end{definition}

\begin{example}\label{ex-smooth}
If the closed subscheme $X$ of $\A^n_{\Z_p}$ defined by the
equation $f=0$ is smooth over $\Z_p$, then the Poincar\'e series
$P(T)$ is easy to compute. For every integer $m\geq 0$, the set
$S_m$ is the set of $(\Z/p^{m+1}\Z)$-valued points on $X$. Locally
at every point, $X$ admits an \'etale morphism to
$\A^{n-1}_{\Z_p}$. The infinitesimal lifting criterion for \'etale
morphisms
 implies that the map $S_{m+1}\to S_m$ is surjective, and that every
fiber has cardinality $p^{n-1}$. In this way, we find that
\begin{equation}\label{eq-smooth}
P(T)=\frac{\sharp X(\mathbb{F}_p)}{1-p^{n-1}T}.\end{equation}
\end{example}

 If
$X$ is not smooth over $\Z_p$, then the behaviour of the values
$N_m$ is much harder to understand. The following conjecture was
mentioned in \cite{borevich-shafarevich}, Chapter 1, Section 5,
Problem 9.

\begin{conjecture}\label{conj-rat}
The Poincar\'e series $P(T)$ is rational, i.e., it belongs to the
subring $\Q(T)\cap \Z[[T]]$ of $\Q((T))$.
\end{conjecture}

\subsection{The $p$-adic zeta function}
\begin{definition}
 We denote by $|\cdot|_p$ the $p$-adic absolute value on $\Q_p$.
The $p$-adic zeta function of $f$ is defined by
$$Z^p_f(s)=\int_{\Z_p^n}|f(x)|_p^s\, dx$$ for every complex number $s$
with $\Re(s)>0$.
\end{definition}

The $p$-adic zeta function $Z^p_f$ is an analytic function on the
complex right half plane
 $\Re(s)>0$.
 It was introduced by Weil, and
systematically studied by Igusa.  It can be defined in a much more
general set-up, starting from a $p$-adic field $K$, an analytic
function $f$ on $K^n$, a Schwartz-Bruhat function $\Phi$ on $K^n$
and a character $\chi$ of $\mathcal{O}_K^{\times}$. Moreover, one
can formulate analogous definitions over the archimedean local
fields $\R$ and $\C$. For a survey, we refer to
\cite{Denef-bourbaki} or \cite{Igusa00}.

We can write $Z^p_f(s)$ as a power series in $p^{-s}$, in the
following way:
$$Z^p_f(s)=\sum_{m\geq 0} \mu_{\mathrm{Haar}}\{a\in
(\Z_p)^n\,|\,v_p(f(a))=m\}p^{-ms}$$ where $v_p$ denotes the
$p$-adic valuation on $\Z_p$. Direct computation shows that, if we
set $T=p^{-s}$, then $Z^p_f(s)$ is related to the Poincar\'e
series $P(T)$ by the formula
\begin{equation}\label{eq-poinzeta}
P(p^{-n}T)=\frac{p^{n}(1-Z^p_f(s))}{1-T}
\end{equation}
Thus the zeta function $Z^p_f(s)$ contains exactly the same
information as the Poincar\'e series $P(T)$, namely, the values
$N_m$ for all $m\geq 0$.

\begin{example}\label{ex-zetasmooth}
In the set-up of Example \ref{ex-smooth}, we have
$$Z^p_f(s)=1-\sharp X(\mathbb{F}_p)p^{-(n-1)}\left( \frac{p^s-1}{p^{s+1}-1} \right).$$
\end{example}

\begin{theorem}[Igusa \cite{Igusa74,Igusa75}]\label{thm-rat}
The $p$-adic zeta function $Z_f^p(s)$ is rational in the variable
$p^{-s}$. In particular, it admits a meromorphic continuation to
$\C$.
\end{theorem}
By \eqref{eq-poinzeta}, this gives an affirmative answer to
Conjecture \ref{conj-rat}:
\begin{cor}
The Poincar\'e series $P(T)$ is rational.
\end{cor}

Igusa proved Theorem \ref{thm-rat} by taking an embedded
resolution of singularities for the zero locus of $f$ in the
$p$-adic manifold $\Z_p^n$, and applying the change of variables
formula for $p$-adic integrals to compute the $p$-adic zeta
function locally on the resolution space. This essentially reduces
the problem to the case where $f$ is a monomial, in which case one
can make explicit computations.

The poles of $P(T)$, or equivalently, $Z^p_f(s)$, contain
information about the asymptotic behaviour of $N_m$ as $m\to
\infty$. Igusa's proof shows that there exists a finite subset
$\mathscr{S}^p$ of $\Q_{<0}$ such that the set of poles of
$Z^p_f(s)$ is given by
$$\{\alpha+\frac{2\pi i}{\ln p}\beta\,|\,\alpha \in
\mathscr{S}^p,\ \beta \in \Z\}.$$ By Denef's explicit formula for
the $p$-adic zeta function in \cite{denef-form}, one can associate
to every embedded resolution for $f$ over $\Q$ a finite subset
$\mathscr{S}$ of $\Q_{<0}$ such that $\mathscr{S}^p\subset
\mathscr{S}$ for $p\gg 0$. The set $\mathscr{S}$ is computed from
the so-called {\em numerical data} of the resolution (in the
notation of \cite{denef-form}, $\mathscr{S}$ is the set of values
$-\nu_i/N_i$ with $i$ in $T$). In general, many of the elements in
$\mathscr{S}$ are not poles of $Z^p_f(s)$, due to cancellations in
the formula for the zeta function. This phenomenon is related to
the {\em Monodromy Conjecture}.

\subsection{Igusa's monodromy conjecture}
Example \ref{ex-zetasmooth} suggests that the poles of the zeta
function $Z^p_f(s)$ should be related to the singularities of the
polynomial $f$. The relation is made precise by Igusa's Monodromy
Conjecture.
\begin{conjecture}[Igusa's Monodromy Conjecture, strong
form]\label{conj-strongmc} If we denote by $b_f(s)$ the
Bernstein-Sato polynomial of $f$, then for $p\gg 0$, the function
$b_f(s)Z^p_f(s)$ is holomorphic at every point of $\R$.
\end{conjecture}
In other words, the conjecture states that for every pole $\alpha$
of $Z^p_f(s)$, the real part $\Re(\alpha)$ is a root of the
Bernstein-Sato polynomial $b_f(s)$, and the order of the pole is
at most the multiplicity of the root.  The Monodromy Conjecture
describes in a precise way how the singularities of $f$ influence
the asymptotic behaviour of the values $N_m$ as $m\to \infty$, for
$p\gg 0$.

\begin{example}
Assume that the closed subscheme of $\A^n_{\Q}$ defined by the
equation $f=0$ is smooth over $\Q$. Then the Bernstein-Sato
polynomial $b_f(s)$ is equal to $s+1$. For $p\gg 0$, the closed
subscheme of $\A^n_{\Z_p}$ defined by $f=0$ is smooth, so that the
zeta function $Z^p_f(s)$ has a unique real pole at $s=-1$, of
order one, by Example \ref{ex-zetasmooth}.
\end{example}

 Because of Kashiwara and Malgrange's result mentioned in Section \ref{subsec-bernstein}, Conjecture \ref{conj-strongmc} implies the following weaker
statement.

\begin{conjecture}[Igusa's Monodromy Conjecture, weak
form]\label{conj-weakmc} For $p\gg 0$, the following holds: if
$\alpha$ is a pole of $Z^p_f(s)$, then $\exp(2\pi \Re(\alpha)i)$
is an eigenvalue of the monodromy action on
$R^j\psi_{f^{\an}}(\C)_x$, for some integer $j\geq 0$ and some
point $x$ of $\C^n$ with $f^{\an}(x)=0$.
\end{conjecture}

Several special cases of the Monodromy Conjecture have been
proven, but the general case remains wide open. For a survey of
known results and the relation with archimedean zeta functions
over the local fields $\R$ and $\C$, we refer to \cite{Ni-japan}.
\subsection{The motivic zeta function}\label{subsec-motzeta}
In the nineties, Denef and Loeser defined a ``motivic'' object
$Z^{\mathrm{mot}}_f(s)$ that interpolates the $p$-adic zeta
functions for $p\gg 0$. It captures the geometric nature of the
$p$-adic zeta functions and explains their uniform behaviour in
$p$. Denef and Loeser called $Z^{\mathrm{mot}}_f(s)$ the {\em
motivic zeta function} associated to $f$. They showed that it is
rational over an appropriate ring of coefficients, and they
conjectured that its poles correspond to roots of the
Bernstein-Sato polynomial as in Conjecture \ref{conj-strongmc}. We
will refer to this conjecture as the Motivic Monodromy Conjecture.
It will be discussed in more detail in Section
\ref{subsec-motmoncon}. For a survey on motivic integration and
motivic zeta functions, and the precise relation with $p$-adic
zeta functions, we refer to \cite{Ni-japan}.

Denef and Loeser defined the motivic zeta function by measuring
spaces of the form
\begin{equation}\label{eq-trunc}\{\psi\in
(k[[t]]/t^{m+1})^n\,|\,f(\psi)\equiv 0 \mod
t^{m+1}\}\end{equation} with $m\geq 0$ and $k$ a field of
characteristic zero. In contrast with the $p$-adic case, the set
\eqref{eq-trunc} is no longer finite, because $k((t))$ is not a
local field. Thus we cannot simply count points in
\eqref{eq-trunc}. Instead, one shows that one can interpret
\eqref{eq-trunc} as the set of $k$-points on an algebraic
$\Q$-variety, and one uses the {\em Grothendieck ring of
varieties} to measure the size of an algebraic variety (see
Section \ref{subsec-K0}).

In the following section, we will explain an alternative
interpretation of the motivic zeta function, due to J. Sebag and
the second author \cite{NiSe-Inv}\cite{Ni-trace}, based on Loeser
and Sebag's theory of motivic integration on non-archimedean
analytic spaces \cite{motrigid}. This interpretation will
eventually lead to the definition of the motivic zeta function of
an abelian variety and, more generally, a Calabi-Yau variety over
a complete discretely valued field.
\section{Motivic integration on rigid varieties and the analytic Milnor fiber}\label{sec-motrig}
\subsection{The Grothendieck ring of varieties}\label{subsec-K0}
Let $F$ be a field.  We denote by $\KVar{F}$ the {\it Grothendieck
ring of varieties over} $F$. As an abelian group, $\KVar{F}$ is
defined by the following presentation:
\begin{itemize}
\item {\em generators:} isomorphism classes $[X]$ of separated
$F$-schemes of finite type $X$,
 \item {\em relations:} if $X$ is
a separated $F$-scheme of finite type and $Y$ is a closed
subscheme of $X$, then
$$[X]=[Y]+[X\setminus Y].$$
These relations are called {\em scissor relations}.
\end{itemize}
By the scissor relations, one has $[X]=[X_{\redu}]$  for every
separated $F$-scheme of finite type $X$, where $X_{\redu}$ denotes
the maximal reduced closed subscheme of $X$. We endow the group
$\KVar{F}$ with the unique ring structure such that
$$[X]\cdot [X']=[X\times_F X']$$ for all $F$-varieties $X$
and $X'$. The identity element for the multiplication is the class
$[\Spec F]$ of the point. We denote by $\LL$ the class $[\A^1_F]$
of the affine line, and by $\mathcal{M}_F$ the localization of
$\KVar{F}$ with respect to $\LL$.

The scissor relations allow to cut an $F$-variety into
subvarieties. For instance, we have
$$[\Proj^2_F]=\LL^2+\LL+1$$ in $\KVar{F}$. Since these are the
only relations that we impose on the isomorphism classes of
$F$-varieties, taking the class of a variety in the Grothendieck
ring should be viewed as the most general way to measure the size
of the variety.

For technical reasons, we'll also need to consider
 the {\em modified Grothendieck ring of
$F$-varieties} $\Kmod{F}$ \cite[\S\,3.8]{NiSe-K0}. This is the
quotient of $\KVar{F}$ by the ideal $\mathcal{I}_F$ generated by
elements of the form $[X]-[Y]$ where $X$ and $Y$ are separated
$F$-schemes of finite type such that there exists a finite,
surjective, purely inseparable $F$-morphism $Y\to X.$

If $F$ has characteristic zero, then it is easily seen that
$\mathcal{I}_F$ is the zero ideal \cite[3.11]{NiSe-K0}. It is not
known if $\mathcal{I}_F$ is non-zero if $F$ has positive
characteristic. In particular, if $F'$ is a non-trivial finite
purely inseparable extension of $F$, it is not known whether
$[\Spec F']\neq 1$ in $\KVar{F}$. With slight abuse of notation,
we'll again denote by $\LL$ the class of $\A^1_F$ in $\Kmod{F}$.
We denote by $\mathcal{M}^{\mathrm{mod}}_F$ the localization of
$\Kmod{F}$ with respect to $\LL$.

For a detailed survey on the Grothendieck ring of varieties and
some intriguing open questions, we refer to \cite{NiSe-K0}.
\subsection{Motivic integration on rigid
varieties}\label{subsec-motint} Let $R$ be a complete discrete
valuation ring, with quotient field $K$ and perfect residue field
$k$. We fix an absolute value on $K$ by assigning a value
$|\pi|\in \,]0,1[$ to a uniformizer $\pi$ of $R$.
 If $R$ has equal
characteristic, then we set $\mathcal{M}_k^R=\mathcal{M}_k$. If
$R$ has mixed characteristic, we set
$\mathcal{M}^R_k=\mathcal{M}^{\mathrm{mod}}_k$.

If $\X$ is a formal $R$-scheme of finite type, then we denote by
$\X_s=\X\times_R k$ its special fiber (this is a $k$-scheme of
finite type) and by $\X_\eta$ its generic fiber (this is a
quasi-compact and quasi-separated rigid $K$-variety; see
\cite{raynaud} or \cite{formrigI}).

\begin{definition}\label{def-bounded}
A rigid $K$-variety $X$ is called {\em bounded} if there exists a
quasi-compact
 open subvariety $U$ of $X$ such that $U(K')=X(K')$ for all finite
 unramified extensions $K'$ of $K$.
\end{definition}

\begin{definition}
Let $X$ be a rigid $K$-variety. A weak N\'eron model for $X$ is a
 smooth formal $R$-scheme of finite type $\X$, endowed
with an open immersion $\X_\eta\to X$, such that
$\X_\eta(K')=X(K')$ for all finite unramified extensions $K'$ of
$K$.
\end{definition}
Note that, if $X$ is separated, then $\X$ will be separated by
\cite[4.7]{formrigI}.

\begin{theorem}[Bosch-Schl\"oter]
A quasi-separated smooth rigid $K$-variety $X$ is bounded if and
only if $X$ admits a weak N\'eron model.
\end{theorem}
\begin{proof}
Since the generic fiber of a formal $R$-scheme of finite type is
quasi-compact, it is clear that the existence of a weak N\'eron
model implies that $X$ is bounded. The converse implication is
\cite[3.3]{bosch-neron}.
\end{proof}
\begin{prop}
Let $X$ be a bounded quasi-separated smooth rigid $K$-variety, and
let $U$ be as in Definition \ref{def-bounded}. If $\X$ is a
regular formal $R$-model of $U$, then the $R$-smooth locus
$\Sm(\X)$ (endowed with the open immersion
$\Sm(\X)_\eta\hookrightarrow \X_\eta=U\hookrightarrow X$) is a
weak N\'eron model for $X$.
\end{prop}
\begin{proof}
If $R'$ is a finite unramified extension of $R$, with quotient
field $K'$, then the specialization map $\X_\eta\to \X$ induces a
bijection $\X_\eta(K')=\X(R')$.
 Every $R'$-point on
$\X$ factors through $\Sm(\X)$, by \cite[2.37]{NiSe-motint}.
\end{proof}

A weak N\'eron model is far from unique, in general, as is
illustrated by the following example.

\begin{example}\label{ex-bounded}
Consider the open unit disc
$$B(0,1^{-})=\{z\in \Sp K\{x\}\,|\ |x(z)|<1
\}.$$ Let $\pi$ be a uniformizer in $R$ and $K'$
 a finite unramified extension of $K$. Then all $K'$-points in $B(0,1^{-})$ are contained in the closed disc
$$B(0,|\pi|)=\{z\in \Sp K\{x\}\,|\ |x(z)|\leq |\pi|\,\}$$
 because $|\pi|$ is the largest
element in the value group $|(K')^*|=|K^*|=|\pi|^{\Z}$ that is
strictly smaller than one. It follows that $B(0,1^{-})$ is
bounded, and that $\X=\Spf R\{u\}$ is a weak N\'eron model for
$B(0,1^{-})$ with respect to the open immersion
$$\X_\eta=\Sp K\{u\}\to B(0,1^{-})$$ defined by $x\mapsto \pi^{-1}u.$
This weak N\'eron model is not unique: one could also remark that
all the unramified points in $B(0,1^{-})$ lie on the union of the
circle $|x(z)|=|\pi|$ and the closed disc $B(0,|\pi|^2)$. In this
way, we get a weak N\'eron model $\X'$ that is the disjoint union
of $\Spf R\{u\}$ and $\Spf R\{v,v^{-1}\}$. Note that $\X'$ can be
obtained by blowing up $\X$ at the origin of $\X_s$ and taking the
$R$-smooth locus.


The open annulus
$$A(0;0^{+},1^{-})=\{z\in \Sp K\{x\}\,|\,0<|x(z)|<1
\}.$$ is not bounded, since $K$-points can lie arbitrarily close
to zero.
\end{example}

Let $X$ be a smooth rigid $K$-variety of pure dimension $m$, and
assume that $X$ admits a weak N\'eron model $\X$. Let $\omega$ be
a gauge form on $X$, i.e., a nowhere vanishing differential form
of maximal degree. Then for every connected component $C$ of
$\X_s$, we can consider the order $\ord_C\omega$ of $\omega$ along
$C$. It is the unique integer $\gamma$ such that
$\pi^{-\gamma}\omega$ extends to a generator of $\Omega^m_{\X/R}$
at the generic point of $C$. In geometric terms, it is the order
of the zero or minus the order of the pole of the form $\omega$
along $C$.

\begin{theorem-definition}[Loeser-Sebag]
Let $X$ be a separated, smooth and bounded rigid $K$-variety of
pure dimension $m$, and let $\X$ be a weak N\'eron model for $X$.
Let $\omega$ be a gauge form on $X$. Then the expression
\begin{equation}\label{eq-motint}
\int_{X}|\omega|:=\LL^{-m}\sum_{C\in
\pi_0(\X_s)}[C]\LL^{-\ord_C\omega}\quad \in
\mathcal{M}^R_k\end{equation} only depends on $(X,\omega)$, and
not on the choice of weak N\'eron model $\X$. We call it the {\em
motivic integral} or {\em motivic volume} of $\omega$ on $X$.
\end{theorem-definition}
\begin{proof}
This is a slight generalization of the result in
\cite[4.3.1]{motrigid}. A proof can be found in
\cite[2.3]{HaNi-zeta}.
\end{proof}

 In this way, we can measure the space of
 unramified points on a bounded separated smooth rigid $K$-variety
 $X$ with respect to a motivic measure defined by a gauge form
 $\omega$ on $X$. Intuitively, one can view the set of unramified
 points on $X$ as a family of open balls parameterized by the special
 fiber of a weak N\'eron model. The gauge form $\omega$
 renormalizes the volume of each ball in such a way that the total
 volume of the family is independent of the chosen model.
  We refer to \cite{NiSe-motint} for more background and
further results.

\subsection{The algebraic case}
One can also define the notion of weak N\'eron model in the
algebraic setting. Let $K^s$ be a separable closure of $K$. Denote
by $R^{\sh}$ the strict henselization of $R$ in $K^s$, and by
$K^{\sh}$ its quotient field. The residue field $k^s$ of $R^{\sh}$
is a separable closure of $k$.

\begin{definition}
Let $X$ be a smooth algebraic $K$-variety. A weak N\'eron model is
a smooth $R$-variety $\X$ endowed with an isomorphism
$$\X\times_R K\to X$$ such that the natural map
\begin{equation}\label{eq-weakneralg}
\X(R^{\sh})\to \X(K^{\sh})=X(K^{\sh})\end{equation} is a
bijection.
\end{definition}
Note that any $k^s$-point on $\X_s$ lifts to an $R^{\sh}$-point on
$\X$, because $\X$ is smooth and $R^{\sh}$ is henselian. Thus
$\X_s$ is empty if and only if $X(K^{\sh})$ is empty.

\begin{remark}
 Since $R^{\sh}$ is the direct limit of all finite unramified
 extensions of $R$ inside $K^{s}$, and $\X$ is of finite type over $R$, we have
 that \eqref{eq-weakneralg} is a bijection if and only if
 $\X(R')\to X(K')$ is a bijection for every finite unramified
 extension $R'$ of $R$. Here $K'$ denotes the quotient field of
 $R'$.
 \end{remark}

\begin{prop}\label{prop-comparweakner}
Let $X$ be a smooth algebraic $K$-variety. Then $X$ admits a weak
N\'eron model $\X$ iff the rigid analytification $X^{\rig}$ admits
a weak N\'eron model, i.e., iff $X^{\rig}$ is bounded. In that
case, the formal $\mathfrak{m}$-adic completion of $\X$ is a weak
N\'eron model for $X^{\rig}$.
\end{prop}
\begin{proof}
This follows from \cite[3.15, 4.3 and 4.9]{Ni-tracevar}.
\end{proof}

In particular, if $X$ is proper over $K$, then $X^{\rig}$ is
quasi-compact, so that $X$ admits a weak N\'eron model.

If $X$ is a smooth $K$-variety with weak N\'eron model $\X$, and
$\omega$ is a gauge form on $X$, then one can define the order
$\ord_C\omega$ of $\omega$ along a connected component $C$ of
$\X_s$ exactly as in the formal-rigid case.

\begin{definition}
Let $X$ be a smooth algebraic $K$-variety of pure dimension such
that the rigid analytification $X^{\rig}$ of $X$ is bounded. Let
$\omega$ be a gauge form on $X$, and denote by $\omega^{\rig}$ the
induced gauge form on $X^{\rig}$. Then we set
$$\int_{X}|\omega|=\int_{X^{\rig}}|\omega^{\rig}|\quad \in \mathcal{M}^R_k.$$
\end{definition}
 By Proposition \ref{prop-comparweakner}, the
motivic integral of $\omega$ on $X$ can also be computed on a weak
N\'eron model of $X$:

\begin{prop}\label{prop-motintalg}
Let $X$ be a smooth algebraic $K$-variety of pure dimension $m$,
and assume that $X$ admits a weak N\'eron model $\X$. For every
gauge form $\omega$ on $X$, we have
$$\int_{X}|\omega|=\LL^{-m}\sum_{C\in \pi_0(\X_s)}[C]\LL^{-\ord_C\omega}$$
in $\mathcal{M}_k^R$.
\end{prop}
\subsection{The analytic Milnor fiber}\label{subsec-milnor}
Let $k$ be any field, and set $R=k[[t]]$ and $K=k((t))$. We fix a
$t$-adic absolute value on $K$ by choosing a value $|t|\in
\,]0,1[$.

Let $X$ be a $k$-variety, endowed with a flat morphism
$$f:X\to \A^1_k=\Spec k[t].$$ Let $x$ be a closed point of the
special fiber $X_s=f^{-1}(0)$ of $f$. Taking the completion of $f$
at the point $x$, we obtain a morphism of formal schemes
\begin{equation}\label{eq-completion}
\widehat{f}_x:\Spf \widehat{\mathcal{O}}_{X,x}\to \Spf
R.\end{equation} We consider the generic fiber $\mathscr{F}_x$ of
$\widehat{f}_x$ in the sense of Berthelot \cite{bert}. This is a
separated rigid variety over the non-archimedean field $K$. It is
bounded, by \cite[5.8]{NiSe-weilres}. If $f$ is generically smooth
(e.g., if $k$ has characteristic zero and $X\setminus X_s$ is
regular) then $\mathscr{F}_x$ is smooth over $K$.

\begin{definition}
We call $\mathscr{F}_x$ the analytic Milnor fiber of $f$ at the
point $x$.
\end{definition}

Note that the construction of the analytic Milnor fiber is a
non-archimedean analog of the definition of the classical Milnor
fibration in complex singularity theory. For an explicit
dictionary, see \cite[6.1]{NiSe-motint}.

\begin{example}
Assume that $x$ is $k$-rational and that $X$ is smooth over $k$ at
$x$. By \cite[19.6.4]{ega4.1}, there exists an isomorphism of
$k$-algebras
$$\widehat{\mathcal{O}}_{X,x}\cong k[[x_1,\ldots,x_n]]$$ where $n$
is the dimension of $X$ at $x$. Viewing $f$ as an element of
$k[[x_1,\ldots,x_n]]$, it defines an analytic function on the open
unit polydisc
$$B^{n}(0,1^{-})=\{z\in \Sp K\{x_1,\ldots,x_n\}\,|\ |x_i(z)|<1
\mbox{ for all }i\}.$$ The analytic Milnor fiber $\mathscr{F}_x$
is the closed subvariety of $B^{n}(0,1^{-})$ defined by the
equation $f=t$. Note that, for every finite unramified extension
$K'$ of $K$, the $K'$-points on $\mathscr{F}_x$ are all contained
in the closed polydisc
$$B^{n}(0,|\pi|)=\{z\in \Sp K\{x_1,\ldots,x_n\}\,|\ |x_i(z)|\leq
|t| \mbox{ for all }i\}$$ by the same argument as in Example
\ref{ex-bounded}. This shows that $\mathscr{F}_x$ is bounded.
\end{example}

The following result follows immediately from \cite[1.3 and
3.5]{berk-vanish2}.
\begin{theorem}[Berkovich]
Assume that $k$ is algebraically closed. Let $\ell$ be a prime
invertible in $k$, and denote by $K^s$ a separable closure of $K$.
Then for every integer $i\geq 0$, there exists a canonical
$G(K^s/K)$-equivariant isomorphism
\begin{equation}\label{eq-comp}H^i(\mathscr{F}_x\widehat{\times}_K
\widehat{K^s},\Q_\ell)\cong R^i\psi_f(\Q_\ell)_x.\end{equation}
\end{theorem}
In the left hand side of \eqref{eq-comp}, we take Berkovich's
\'etale cohomology for $K$-analytic spaces \cite{berk-etale}. In
the right hand side, $R\psi_f(\Q_\ell)$ denotes the complex of
\'etale $\ell$-adic nearby cycles associated to $f$. If $k=\C$,
then by Deligne's comparison theorem \cite[Exp.XIV]{sga7b}, there
exists a canonical isomorphism
$$R^i\psi_f(\Q_\ell)_x\cong R^i\psi_{f^{\mathrm{an}}}(\Q_\ell)_x$$
 where $f^{\mathrm{an}}:X^{\mathrm{an}}\to \C$ is the complex
 analytification of $f$ and $\psi_{f^{\mathrm{an}}}$ is the
 complex analytic nearby cycles functor. Under this isomorphism,
 the action of the canonical generator of
 $G(K^s/K)=\mu(\C)$ on $R^i\psi_f(\Q_\ell)_x$ corresponds to the monodromy
 transformation on $R^i\psi_{f^{\mathrm{an}}}(\Q_\ell)_x$. This
 means that we can read the local monodromy eigenvalues of $f^{\mathrm{an}}$ at
 $x$ from the \'etale cohomology of the analytic Milnor fiber.

The following proposition shows that the analytic Milnor fiber
$\mathscr{F}_x$ completely determines the formal germ
$\widehat{f}_x$ of $f$ at $x$, if $X$ is normal at $x$. We denote
by $\mathcal{O}(\mathscr{F}_x)$ the $K$-algebra of analytic
functions on $\mathscr{F}_x$.

\begin{prop}[de Jong \cite{dJ}, Prop. 7.4.1; see also \cite{Ni-trace}, Prop.
8.8]\label{prop-deJong} If $X$ is normal at $x$, then there exists
a natural isomorphism of $R$-algebras
$$\widehat{\mathcal{O}}_{X,x}\cong \{h\in
\mathcal{O}(\mathscr{F}_x)\,|\ |h(z)|\leq 1\mbox{ for all }z\in
\mathscr{F}_x\}.$$
\end{prop}
There is an interesting relation between the Berkovich topology on
$\mathscr{F}_x$ and the limit mixed Hodge structure on the nearby
cohomology of $f$ at $x$; see \cite{Ni-sing}.

\subsection{The Gelfand-Leray form}\label{subsec-gelfandleray} We keep the notations of
Section \ref{subsec-milnor}, and we assume that $k$ has
characteristic zero and that $X$ is smooth over $k$ at the point
$x$. Replacing $X$ by an open neighbourhood of $x$, we can assume
that $f$ is smooth on $X^o=X\setminus X_s$, of relative dimension
$m$, and that $\Omega^{m+1}_{X^o/k}$ is free, i.e., that $X$ admits
a gauge form $\phi$. Then the exact complex
$$\begin{CD}\Omega^{m-1}_{X^o/k}@>\wedge df>> \Omega^{m}_{X^o/k}@>\wedge df>> \Omega^{m+1}_{X^o/k}@>>> 0 \end{CD}$$
induces an isomorphism of sheaves
$$\Omega^m_{X^o/\A^1_k}\to \Omega^{m+1}_{X^o/k}$$ and thus an
isomorphism of $\mathcal{O}(X^o)$-modules
$$\Omega^m_{X^o/\A^1_k}(X^o)\to \Omega^{m+1}_{X^o/k}(X^o).$$ The
inverse image in $\Omega^m_{X^o/\A^1_k}(X^o)$ of the restriction
of $\phi$ to $X^o$ is called the {\em Gelfand-Leray form} on $X^o$
associated to $f$ and $\phi$. It is denoted by $\phi/df$. It
induces a gauge form on the analytic Milnor fiber $\mathscr{F}_x$,
since $\mathscr{F}_x$ is an open rigid sub-$K$-variety of the
rigid analytification of $X\times_{k[t]}K$ \cite[0.2.7 and
0.3.5]{bert}. We denote this gauge form again by $\phi/df$. It can
be constructed intrinsically on $\mathscr{F}_x$; see Proposition
\ref{prop-deJong} and \cite[\S\,7.3]{Ni-trace}.
\subsection{The motivic zeta function and the motivic monodromy conjecture}\label{subsec-motmoncon} We keep the notations of
Section \ref{subsec-milnor}. We assume that $k$ has characteristic
zero, and that $X$ is smooth at $x$, of dimension $n$. For
simplicity, we suppose that $x$ is $k$-rational. Let $\phi$ be a
gauge form on some open neighbourhood of $x$ in $X$, and consider
the Gelfand-Leray form $\phi/df$ on $\mathscr{F}_x$ constructed in
Section \ref{subsec-gelfandleray}. We set $\omega=t\cdot \phi/df$.
This is a gauge form on $\mathscr{F}_x$. For every integer $d>0$,
we set $K(d)=k((\sqrt[d]{t}))$. This is a totally ramified
extension of $K$ of degree $d$. We set
$\mathscr{F}_x(d)=\mathscr{F}_x\times_K K(d)$. For every
differential form $\omega'$ on $\mathscr{F}_x$, we denote by
$\omega'(d)$ the pullback of $\omega'$ to $\mathscr{F}_x(d)$.

 We denote by $Z_{f,x}(T)\in \mathcal{M}_k[[T]]$ Denef and Loeser's local motivic zeta
 function of $f$ at the point $x$ (obtained from the zeta function in \cite[3.2.1]{DL-geom} by taking the fiber at $x$ and forgetting the
 $\mu$-action).
The reader who is unfamiliar with Denef and Loeser's definition
may take the following theorem as a definition.
\begin{theorem}[Nicaise-Sebag]\label{thm-compar}
We have \begin{equation}\label{eq-comparzeta}
Z_{f,x}(T)=\LL^{n-1}\sum_{d>0}\left(\int_{\mathscr{F}_x(d)}|\omega(d)|\right)T^d
\end{equation} in $\mathcal{M}_{k}[[T]]$.
\end{theorem}
\begin{proof}
 Note that, for every $d$ in $\N$, we have
$$ \int_{\mathscr{F}_x(d)}|\omega(d)|=\LL^{-d}\int_{\mathscr{F}_x(d)}|(\phi/df)(d)| $$ in $\mathcal{M}_k$,
since $t$ has valuation $d$ in $K(d)$. Thus the theorem is a
reformulation of \cite[9.7]{Ni-trace}, which is a consequence of
the comparison theorem in \cite[9.10]{NiSe-Inv}.
\end{proof}

The proof of Theorem \ref{thm-compar} is based on an explicit
construction of weak N\'eron models for the rigid varieties
$\mathscr{F}_x(d)$, starting from an embedded resolution of
singularities for $(X,X_s)$. In this way, one obtains an explicit
formula for the right hand side in \eqref{eq-comparzeta} in terms
of such a resolution, and one can compare this expression to the
formula for $Z_{f,x}(T)$ obtained by Denef and Loeser
\cite[3.3.1]{DL-geom}. This formula implies in particular that
$Z_{f,x}[T]$ is contained in the subring
$$\mathcal{M}_k\left[T,\frac{1}{1-\LL^a T^b}\right]_{a\in \Z_{<0},\,b\in \Z_{>0}}$$ of
$\mathcal{M}_k[[T]]$.

If the residue field $k_x$ of $x$ is not $k$, one can adapt the
construction as follows. Since $k$ has characteristic zero, $k_x$
is a separable extension of $k$ so that the morphism
$\widehat{f}_x$ from \eqref{eq-completion} factors through a
morphism
$$\widehat{\mathcal{O}}_{X,x}\to \Spf k_x[[t]]$$ by \cite[19.6.2]{ega4.1}. In this way, we
can view $\mathscr{F}_x$ as a rigid variety over $K_x=k_x((t))$.
Since $K_x$ is separable over $K$, the natural morphism
$\Omega^i_{\mathscr{F}_x/K_x}\to \Omega^i_{\mathscr{F}_x/K}$ is an
isomorphism for all $i$, so that we can consider the Gelfand-Leray
form $\phi/df$ as an element of  $\Omega^{m}_{\mathscr{F}_x/K_x}$.
Then the equality \eqref{eq-comparzeta} holds in
$\mathcal{M}_{k_x}[[T]]$.

\begin{conjecture}[Motivic Monodromy
Conjecture]\label{conj-motmon}  Assume that $k$ is a subfield of
$\C$. There exists a finite subset $\mathscr{S}$ of $\Z_{<
0}\times \Z_{>0}$ such that $Z_{f,x}(T)$ belongs to the subring
$$\mathcal{M}_{k_x}\left[T,\frac{1}{1-\LL^a T^b}\right]_{(a,b)\in
\mathscr{S}}$$ of $\mathcal{M}_{k_x}[[T]]$, and such that for
every couple $(a,b)$ in $\mathscr{S}$, the quotient $a/b$ is a
root of the Bernstein-Sato polynomial $b_f(s)$ of $f$. In
particular, there exists a point $y$ of $X(\C)$ such that $f(y)=0$
and such that $\exp(2\pi i a/b)$ is a local monodromy eigenvalue
of $f$ at the point $y$.
\end{conjecture}
\begin{remark}
One can drop the condition that $k$ is a subfield of $\C$ by using
$\ell$-adic nearby cycles to define the notion of local monodromy
eigenvalue. This does not yield a more general conjecture, since
by the Lefschetz principle, one can always reduce to the case
where $k$ is a subfield of $\C$.
\end{remark}

One needs to be careful when speaking about poles of the zeta
function, since $\mathcal{M}_{k_x}$ is not a domain. A precise
definition is given in \cite{RoVe}. The formulation in Conjecture
\ref{conj-motmon} implies that, for any reasonable definition of
pole in this context, the poles of $Z_{f,x}(\LL^{-s})$ are of the
form $a/b$, with $(a,b)\in \mathscr{S}$. With some additional
work, one can define the order of a pole \cite{RoVe}, and
conjecture that the order of a pole of $Z_{f,x}(\LL^{-s})$ is at
most the multiplicity of the corresponding root of $b_f(s)$.

\subsection{Relation with the $p$-adic monodromy
conjecture}\label{subsec-comparmon} Let us explain the precise
relation between Conjectures \ref{conj-motmon} and
\ref{conj-strongmc}. In \cite[3.2.1]{DL-geom}, Denef and Loeser
define the motivic zeta function $Z_f(T)$ (there denoted by
$Z(T)$) associated to the morphism $f$. It carries more structure
than we've considered so far: it is a formal power series with
coefficients in the {\em
 equivariant} Grothendieck ring of $X_s$-varieties
 $\mathcal{M}_{X_s}^{\mu}$. The elements of this ring are virtual classes of
 $X_s$-varieties that carry an action of the profinite group
 scheme $\mu$ of roots of unity. The $X_s$-structure
 allows to consider various ``motivic Schwartz-Bruhat functions''
 and the
 $\mu$-action allows to twist the motivic zeta function
 by ``motivic characters'', like in the $p$-adic case; see
 \cite{motigusa}. The motivic zeta function that we've alluded to
 in Section \ref{subsec-motzeta} corresponds to the trivial character;
 it is the ``na\"{i}ve'' motivic zeta function from
 \cite[3.2.1]{DL-geom}. If $k=\Q$ and $X=\A^n_k$, then for $p\gg 0$, we can
 specialize the na\"{i}ve motivic zeta function of $f$ to the
 $p$-adic one, by counting rational points on the reductions of
 the coefficients modulo $p$. This is explained in \cite[\S5.3]{Ni-japan}.

 The local zeta function
 $Z_{f,x}(T)$ that we've considered above is obtained from
 $Z_f(T)$ by applying the morphism
 $$\mathcal{M}_{X_s}^{\mu}\to \mathcal{M}_{k_x}$$
(base change to $x$ and forgetting the $\mu$-structure)
 to the
 coefficients of $Z_f(T)$. In an unpublished manuscript, the second author has shown that it is possible to recover the
 $\mu$-structure on $Z_{f,x}(T)$ by considering the
 Galois action on the extensions $K(d)$ of $K$.

One can formulate Conjecture \ref{conj-motmon} for
 $Z_f(T)$ instead of $Z_{f,x}(T)$, as follows:
 \begin{conjecture}[Motivic Monodromy
Conjecture II]\label{conj-motmon2}  Assume that $k$ is a subfield
of $\C$. There exists a finite subset $\mathscr{S}$ of
$\Z_{<0}\times \Z_{>0}$ such that $Z_{f}(T)$ belongs to the
subring
$$\mathcal{M}^{\mu}_{X_s}\left[T,\frac{1}{1-\LL^a T^b}\right]_{(a,b)\in
\mathscr{S}}$$ of $\mathcal{M}^{\mu}_{X_s}[[T]]$, and such that
for every couple $(a,b)$ in $\mathscr{S}$, the quotient $a/b$ is a
root of the Bernstein-Sato polynomial $b_f(s)$ of $f$. In
particular, there exists a point $y$ of $X(\C)$ such that $f(y)=0$
and such that $\exp(2\pi i a/b)$ is a local monodromy eigenvalue
of $f$ at the point $y$.
\end{conjecture}
This conjecture implies
 \begin{enumerate}
\item Conjecture \ref{conj-motmon}, since we can specialize
$Z_f(T)$ to $Z_{f,x}(T)$ by taking fibers at $x$ and forgetting
the $\mu$-structure,
 \item Conjecture
\ref{conj-strongmc}, because $Z_f(T)$ can be specialized to the
$p$-adic zeta
 function,
 \item more generally, the $p$-adic
 monodromy conjecture for zeta functions that are twisted by
 characters \cite[\S2.4]{motigusa}.
\end{enumerate}
Various weaker reformulations of Conjecture \ref{conj-motmon2}
have appeared in the literature (e.g. in \cite[\S2.4]{motigusa}).
The formulation we use seems to be part of general folklore. We
attribute it to Denef and Loeser.

\section{The motivic zeta function of an abelian variety}\label{sec-abelian}
\subsection{Some notation}\label{subsec-not}
Let $R$ be a complete discrete valuation ring with maximal ideal
$\mathfrak{m}$, fraction field $K$ and algebraically closed
residue field $k$. We denote by $p$ the characteristic exponent of
$k$, and by $\N'$ the set of strictly positive integers that are
prime to $p$. We fix a prime $\ell\neq p$ and a separable closure
$K^s$ of $K$. The Galois group $G(K^s/K)$ is called the {\em
inertia group} of $K$.

A finite extension of $K$ is called {\em tame} if its degree is
prime to $p$. For every $d$ in $\N'$, the field $K$ admits  a
unique degree $d$ extension $K(d)$ in $K^s$. It is obtained by
joining a $d$-th root of a uniformizer to $K$. The extension
$K(d)/K$ is Galois, with Galois group $\mu_d(k)$.

The union of the fields $K(d)$ is a subfield of $K^s$, called the
{\em tame closure} $K^t$ of $K$. The Galois group $G(K^t/K)$ is
called the {\em tame inertia group} of $K$. It is canonically
isomorphic to the procyclic group
$$\mu'(k)=\lim_{\stackrel{\longleftarrow}{d\in \N'}}\mu_d(k)$$
where the elements in $\N'$ are ordered by divisibility and the
transition morphisms in the projective system are given by
$$\mu_{de}(k)\to \mu_d(k):x\mapsto x^e$$ for all $d,e$ in $\N'$.
 We call every
topological generator of $G(K^t/K)$ a {\em tame monodromy
operator}. The Galois group $P=G(K^s/K^t)$ is a pro-$p$-group
which is called the {\em wild inertia subgroup} of $G(K^s/K)$. We
have a short exact sequence
$$1\to P\to G(K^s/K)\to G(K^t/K)\to 1.$$
\subsection{N\'eron models and semi-abelian reduction}
Let $A$ be an abelian $K$-variety of dimension $g$. It is not
always possible to extend $A$ to an abelian scheme over $R$.
However, there exists a canonical way to extend $A$ to a smooth
commutative group scheme over $R$, the so-called \emph{N\'eron
model} of $A$.

\begin{definition}
A N\'eron model of $A$ is a smooth $R$-scheme of finite type
$\mathcal{A}$, endowed with an isomorphism $$\mathcal{A}\times_R
K\to A,$$ such that the natural map
\begin{equation}\label{eq-Ner}
\Hom_R(T,\mathcal{A}) \to \Hom_K(T \times_R K, A)
\end{equation}
is a bijection for every smooth $R$-scheme $T$.
\end{definition}

Thus $\mathcal{A}$ is the minimal smooth $R$-model of
 $A$. The existence of a N\'eron model was first proved by A. N\'eron
\cite{neron0}. For a modern scheme-theoretic treatment of the
theory and an accessible proof of N\'eron's theorem, we refer to
\cite{neron}.  The universal property of the N\'eron model implies
that the
 N\'eron model
 $\mathcal{A}$ is unique up to unique isomorphism, and that the
 $K$-group structure on $A$ extends uniquely to a commutative $R$-group
 structure on $\mathcal{A}$. Taking for $T$ the spectrum of a
 finite unramified extension of $R$, we see that  $\mathcal{A}$ is also a weak N\'eron
 model for $A$.

The special fiber $\mathcal{A}_s := \mathcal{A} \times_R k$ is a
smooth commutative algebraic $k$-group. We denote by
$\mathcal{A}_s^{o}$ the identity component of $\mathcal{A}_s$,
i.e., the connected component containing the identity point for
the group structure. The quotient $ \Phi_A :=
\mathcal{A}_s/\mathcal{A}_s^{o}$ is called the \emph{component
group}. It is a finite \'etale group scheme over $k$ whose group
of $k$-points corresponds bijectively to the the set of connected
components of $\mathcal{A}_s$. Since $k$ is assumed to be
algebraically closed, we will not distinguish between the group
scheme $ \Phi_A $ and the abstract group $\Phi_A(k)$.

The identity component $\mathcal{A}_s^{o}$ fits into a canonical
short exact sequence of algebraic $k$-groups, the \emph{Chevalley
decomposition},
\begin{equation}\label{eq-Che}
0 \to T \times_k U \to \mathcal{A}_s^{o} \to B \to 0
\end{equation}
where $B$ is an abelian variety, $T$ is a torus and $U$ is a
unipotent group commonly referred to as the \emph{unipotent
radical} of $\mathcal{A}_s^{o}$. We call the dimension of $T$ the
{\em toric rank} of $A$, and the dimension of $U$ the {\em
unipotent rank} of $A$.

\begin{definition}
We say that $A$ has semi-abelian reduction if the unipotent rank
of $\mathcal{A}_s^{o}$ is zero.
\end{definition}

A celebrated result by A. Grothendieck, the Semi-Stable Reduction
Theorem for abelian varieties \cite[IX.3.6]{sga7a}, asserts that
there exists a finite separable extension $K'/K$ such that $A
\times_K K' $ has semi-abelian reduction over the integral closure
$R'$ of $R$ in $K'$. Inside our fixed separable closure $K^s$ of
$K$, there exists a unique minimal extension $L$ with this
property, and it is Galois over $K$. By \cite[IX.3.8]{sga7a}, $L$
is the fixed field of the subgroup of $G(K^s/K)$ consisting of the
elements that act unipotently on the $\ell$-adic Tate module
$T_\ell A$ of $A$. If $L$ is a tame extension of $K$
 then we say
that $A$ is \emph{tamely ramified}. Since the $P$-action on
$T_\ell A$ factors through a finite quotient of $P$
\cite[p.3]{lenstra-oort}, $A$ is tamely ramified if and only if
$P$ acts trivially on $T_\ell A$.
 In that
 case, $P$ acts trivially on the $\ell$-adic cohomology of
$A$, and the natural morphism
$$H^i(A\times_K K^t,\Q_\ell)\to H^i(A\times_K K^s,\Q_\ell)$$ is an
isomorphism for every $i$ in $\N$.


\subsection{The base change conductor and the potential toric rank}\label{subsec-bcc}
In \cite{chai}, Chai introduced an invariant that measures how far
the abelian $K$-variety $A$ is from having semi-abelian reduction.
He called it the {\em base change conductor} of $A$ and denoted it
by $c(A)$. It is defined as follows. Let $K'/K$ be a finite
separable extension such that $A$ acquires semi-abelian reduction
over $K'$, and let $\mathcal{A}'$ be the N\'eron model of $ A
\times_K K' $. By the universal property of the N\'eron model
$\mathcal{A}'$, there is a unique morphism
\begin{equation}\label{eq-h} h : \mathcal{A} \times_R R' \to
\mathcal{A}' \end{equation} that extends the canonical isomorphism
between the generic fibers. The induced map
$$ \Lie(h) : \Lie(\mathcal{A} \times_R R') \to \Lie(\mathcal{A}') $$
is an injective homomorphism of free $R'$-modules of rank $g$, so
that $\coker(\Lie(h))$ is an $R'$-module of finite length. The
rational number
$$ c(A) := [K':K]^{-1} \cdot \length_{R'} (\coker(\Lie(h))) $$
is independent of the choice of $K'$.

The importance of the Semi-Stable Reduction Theorem lies in the
fact that, if $A$ has semi-abelian reduction, then $h$ is an open
immersion, so that it induces an isomorphism between the identity
components of $\mathcal{A}\times_R R'$ and $\mathcal{A}'$
\cite[IX.3.3]{sga7a} (the number of connected components of
$\mathcal{A}$ might still change, though; this will be discussed
below). Thus $c(A)$ vanishes if $A$ has semi-abelian reduction.
Conversely, if $c(A)=0$ then $h$ must be \'etale, and the fact
that $h$ restricts to an isomorphism between the generic fibers
then implies that $h$ is an open immersion. Thus $c(A)$ is zero if
and only if $A$ has semi-abelian reduction.

Another invariant that will be important for us is the toric rank
of $A\times_K K'$. We call this value the {\em potential toric
rank} of $A$, and denote it by $t_{\mathrm{pot}}(A)$. Again, it is
independent of the choice of $K'$. Moreover, it is the maximum of
the toric ranks of the abelian varieties $A\times_K K''$ as $K''$
ranges over all the finite separable extensions of $K$. The
potential toric rank is a measure for the potential degree of
degeneration of $A$ over the closed point of $\Spec R$; it
vanishes if and only if, after a finite separable extension of the
base field $K$, the abelian variety $A$ extends to an abelian
scheme over $R$. In this case, we say that $A$ has {\em potential
good reduction}. If $t_{pot}(A)$ has the largest possible value,
namely, the dimension of $A$, then we say that $A$ has {\em
potential purely multiplicative reduction}. Thus $A$ has potential
purely multiplicative reduction if and only if the identity
component of $\mathcal{A}'_s$ is a torus.

\subsection{The motivic zeta function of an abelian variety}
For every $d$ in $\N'$, we set $A(d) = A \times_K K(d)$ and we
denote by $\mathcal{A}(d)$
 the N\'eron model of $A(d)$. For every gauge form $\omega$ on $A$, we denote by $\omega(d)$ its
pullback to $A(d)$. We define the order
$\mathrm{ord}_{\mathcal{A}(d)^o_s}\omega(d)$ of $\omega(d)$ along
$\mathcal{A}(d)_s^o$ as in Section \ref{subsec-motint}.





\begin{definition}\label{def-disformabelian}
A gauge form $\omega$ on $A$ is  \emph{distinguished} if it is the
restriction to $A$ of a generator of the free rank one
$\mathcal{O}_A$-module $ \Omega^g_{\mathcal{A}/R} $.
\end{definition}

It is clear from the definition that a distinguished gauge form
always exists, and that it is unique up to multiplication with
 a unit in $R$. Note that a gauge form $\omega$ on $A$ is
 distinguished if and only if $\ord_{\mathcal{A}_s^o}\omega=0$.

 In general, a distinguished gauge form on $A$ does not remain
 distinguished under base change to a finite tame extension $K(d)$
 of $K$. To measure the defect, we introduce the following
 definition.

\begin{definition}
Let $A$ be an abelian $K$-variety, and let $\omega$ be a
distinguished gauge form on $A$. The order function of $A$ is the
function
$$\ord_A:\N'\to \N:d\mapsto -\ord_{\mathcal{A}(d)_s^o}\omega(d).$$
\end{definition}
This definition does not depend on the choice of distinguished
gauge form, since multiplying $\omega$ with a unit in $R$ does not
affect the order of $\omega(d)$ along $\mathcal{A}(d)_s^o$. The
fact that $\ord_A$ takes its values in $\N$ follows easily from
the existence of the morphism $h$ in \eqref{eq-h}, for arbitrary
finite extensions $K'$ of $K$.

\begin{definition}\label{def-abzeta}
Let $A$ be an abelian $K$-variety, and let $\omega$ be a
distinguished gauge form on $A$. We define the motivic zeta
function $Z_A(T)$ of $A$ as
$$Z_A(T)=\sum_{d\in
\N'}[\mathcal{A}(d)_s]\LL^{\ord_{A}(d)} T^d \in
\mathcal{M}_k[[T]].$$
\end{definition}
 The
following proposition gives an interpretation of $Z_A(T)$ in terms
of the volumes of the ``motivic Haar measures'' $|\omega(d)|$.

\begin{prop}\label{prop-mothaar}
Let $A$ be an abelian $K$-variety of dimension $g$, and let
$\omega$ be a distinguished gauge form on $A$. The image of
$Z_A(T)$ in the quotient ring $\mathcal{M}^R_k[[T]]$ of
$\mathcal{M}_k[[T]]$ is equal to
$$\LL^g \sum_{d\in \N'}\left(\int_{A(d)}|\omega(d)|\right)T^d.$$
\end{prop}
\begin{proof}
We've already observed that every N\'eron model is also a weak
N\'eron model. The gauge form $\omega$ is translation-invariant,
so that
$$\ord_{C}\omega(d)=\ord_{\mathcal{A}(d)_s^o}\omega(d)$$ for every
connected component $C$ of $\mathcal{A}(d)_s$. Now the result
follows immediately from the definition of the motivic integral,
and the fact that
$$[\mathcal{A}(d)_s]=\sum_{C\in \pi_0(\mathcal{A}(d)_s)}[C]$$ by the
scissor relations in $\mathcal{M}_k$.
\end{proof}

\subsection{The monodromy conjecture} Now we come to the formulation of the main result  of
\cite{HaNi-zeta}, which is a variant of Conjecture
\ref{conj-weakmc} for abelian varieties. For every integer $d\geq
0$, we denote by $\Phi_d(t)\in \Z[t]$ the cyclotomic polynomial
whose roots are the primitive $d$-th roots of unity. For every
rational number $q$, we denote by $\tau(q)$ its order in the group
 $\mathbb{Q}/\mathbb{Z}$.

\begin{theorem}[Monodromy conjecture for abelian varieties]\label{thm-main}
Let $A$ be a tamely ramified abelian variety of dimension $g$, and
let $\sigma$ be a tame monodromy operator in $G(K^t/K)$.
\begin{enumerate}
\item The motivic zeta function $Z_A(T)$ belongs to the subring
$$ \mathcal{M}_k \left[ T, \frac{1}{1 - \mathbb{L}^a T^b}\right]_{(a,b) \in \mathbb{N} \times \mathbb{Z}_{>0}, a/b = c(A)} $$
of $ \mathcal{M}_k[[T]] $. The zeta function
$Z_A(\mathbb{L}^{-s})$ has a unique pole at $ s= c(A) $, of order
$t_{pot}(A) + 1$.

\item The cyclotomic polynomial $\Phi_{\tau(c(A))}(t)$ divides the
characteristic polynomial of the tame monodromy operator $\sigma$
on $H^g(A \times_K K^t, \mathbb{Q}_{\ell}) $. Thus for every
embedding $\mathbb{Q}_{\ell} \hookrightarrow \mathbb{C}$, the
value $ \exp(2 \pi c(A) i) $ is an eigenvalue of $\sigma$ on
$H^g(A \times_K K^t, \mathbb{Q}_{\ell}) $.
\end{enumerate}
\end{theorem}

We'll briefly sketch the main ideas of the proof. The first step
is to refine the expression for the motivic zeta function, as
follows. For every $d$ in $\N'$, we denote by
 $t_A(d)$ and $u_A(d)$ the toric, resp. unipotent rank of $A(d)$, and we denote by $B_A(d)$ the abelian quotient
 in the Chevalley decomposition of
$\mathcal{A}(d)_s^{o}$. Moreover, we denote by $ \phi_A(d) =
|\Phi_{A(d)}| $ the number of connected components of the
$k$-group $\mathcal{A}(d)_s$.

\begin{prop}\label{prop-refine}
For every abelian $K$-variety $A$, we have \begin{eqnarray*}Z_A(T)
&=& \sum_{d \in \mathbb{N}'} [\mathcal{A}(d)_s]
\mathbb{L}^{ord_A(d)} T^d \\ &=& \sum_{d \in \mathbb{N}'} \left(
\phi_A(d) \cdot (\mathbb{L}-1)^{t_A(d)} \cdot \mathbb{L}^{u_A(d) +
ord_A(d)} \cdot [B_A(d)] T^d \right) \end{eqnarray*} in $
\mathcal{M}_k[[T]] $.
\end{prop}
\begin{proof}
The first equality is simply the definition of the zeta function.
For every $d\in \N'$, the connected components of
$\mathcal{A}(d)_s$ are all isomorphic to $\mathcal{A}(d)_s^o$,
because $k$ is algebraically closed. Thus by the scissor relations
in the Grothendieck ring, one has
$$[\mathcal{A}(d)_s]=\phi_A(d)\cdot [\mathcal{A}(d)_s^o].$$
Now consider the Chevalley decomposition
$$0\to T_A(d)\times_k U_A(d)\to \mathcal{A}(d)_s^o\to B_A(d)\to
0$$ of $\mathcal{A}(d)_s^o$. The torus $T_A(d)$ is isomorphic to
$\mathbb{G}_{m,k}^{t_A(d)}$, and $U_A(d)$ is a successive
extension of additive groups $\mathbb{G}_{a,k}$. It follows that
$\mathcal{A}(d)_s^o\to B_A(d)$ is a Zariski-locally trivial
fibration. Moreover, as a $k$-variety, $U_A(d)$ is isomorphic to
$\A^{u_A(d)}_k$. Thus $$[\mathcal{A}(d)_s^{o}] =
(\mathbb{L}-1)^{t_A(d)} \cdot \mathbb{L}^{u_A(d)} \cdot [B_A(d)]
$$ in $ \mathcal{M}_k $.
\end{proof}
Thus the study of $Z_A(T)$ can be split up into the following
subproblems:
\begin{enumerate}
\item How do $t_A(d)$, $u_A(d)$ and $B_A(d)$ vary with $d$? \item
What is the shape of the order function $\ord_A$?  \item How does
$\phi_A(d)$ vary with $d$?
\end{enumerate}

Our main tool in the study of (1) was a theorem due to B.
Edixhoven \cite{Edix}, which says that for every $d\in \N'$, the
N\'eron model $\mathcal{A}$ is canonically isomorphic to the fixed
locus of the $G(K(d)/K)$-action on the Weil restriction of
$\mathcal{A}(d)$ to $R$. This result enabled us to show that
 $t_A(d)$, $u_A(d)$ and $[B_A(d)]$ only depend on the residue class of $d$ modulo $e$, where $e$ is the degree
 of the minimal extension of $K$ where $A$ acquires semi-abelian
 reduction. In the same paper, Edixhoven constructs a filtration
 on $\mathcal{A}_s$ by closed algebraic subgroups, indexed by $\Q\cap [0,1]$. This filtration measures the behaviour of the
 N\'eron model of
 $A$ under tamely ramified base change. The {\em jumps} of $A$ are the indices where the subgroup changes. Edixhoven related
 these
 jumps
 to the Galois action of $G(K(e)/K)=\mu_e(k)$ on the
 $k$-vector space $\Lie(\mathcal{A}(e)_s^o)$.
 We deduced from Edixhoven's
 theory that $c(A)$ is the sum of the jumps of $A$ and that on every residue class of $\N'$ modulo $e$, the function $\ord_A$ is
 affine with slope $c(A)$. The function $\ord_A$ is responsable
 for the pole of $Z_A(\LL^{-s})$ at $s=c(A)$.

To control the behaviour of $\phi_A(d)$ turned out to be rather
involved, we treated this in the separate paper \cite{HaNi-comp}.
There, we used rigid uniformization of $A$ in the sense of
\cite{BoXa} to reduce to the case of tori and abelian varieties
with potential good reduction, where more explicit methods could
be used to describe the change in the component groups under
ramified base extensions. In this way, we obtained the following
result.
\begin{theorem}
Let $A$ be an abelian $K$-variety, and let $e$ be the degree of
the minimal extension of $K$ where $A$ acquires semi-abelian
reduction. Denote by $t(A)$ the toric rank of $A$ and by $\phi(A)$
the number of connected components of the N\'eron model of $A$.
Assume either that $A$ is tamely ramified or that $A$ has
potential purely multiplicative reduction. Then for every element
$d$ of $\N'$ that is prime to $e$, we have
$$\phi_A(d) = \phi(A) \cdot d^{t(A)}.$$
\end{theorem}
This result was sufficient for our purposes. The behaviour of
$\phi_A(d)$ is responsible for the order $t_{\mathrm{pot}}(A)+1$
of the unique pole of the zeta function.

It remains to prove the relation between the base change conductor
and the tame monodromy action on $H^g(A\times_K K^t,\Q_\ell)$.
Here we again used Edixhoven's theory and we showed how to compute
the eigenvalues of $\sigma$ on $H^1(A\times_K K^t,\Q_\ell)$ from
the Galois action of $\mu_e(k)$ on $\Lie(\mathcal{A}(e)_s^o)$.


\subsection{Strong version of the monodromy conjecture}
It is natural to ask for an analog of Conjecture
\ref{conj-strongmc} for abelian varieties. There is no good notion
of Bernstein polynomials in this setting. However, the
multiplicities of the roots of the Bernstein polynomial of a
complex hypersurface singularity are closely related to the sizes
of the Jordan blocks of the monodromy action on the cohomology of
the Milnor fiber, so one may ask if the order of the pole of
$Z_A(T)$ is related to Jordan blocks of the tame monodromy action
on $H^g(A\times_K K^t,\Q_\ell)$. We've shown in
\cite{HaNi-jumpmon} that this is indeed the case.

\begin{theorem}\label{thm-Jor}
Let $A$ be a tamely ramified abelian $K$-variety of dimension $g$.
For every tame monodromy operator $\sigma$ in $G(K^t/K)$ and every
embedding of $\mathbb{Q}_{\ell}$ in $\mathbb{C}$, the value $
\alpha = \exp(2 \pi c(A) i) $ is an eigenvalue of $\sigma$ on
$H^g(A \times_K K^t, \mathbb{Q}_{\ell}) $. Each Jordan block of
$\sigma$ on $H^g(A \times_K K^t, \mathbb{Q}_{\ell}) $ has size at
most $t_{pot}(A) + 1$, and $\sigma$ has a Jordan block with
eigenvalue $ \alpha $ on $H^g(A \times_K K^t, \mathbb{Q}_{\ell}) $
with size $t_{pot}(A) + 1$.
\end{theorem}
 In the case $K=\C((t))$, we also gave in \cite{HaNi-jumpmon} a Hodge-theoretic
 interpretation of the jumps in Edixhoven's filtration, in
 terms of the limit mixed Hodge structure associated to $A$.

\subsection{Cohomological interpretation}
The motivic zeta function of an abelian $K$-variety admits a
cohomological interpretation, by \cite{Ni-abelian}. We consider
the unique ring morphism
$$\chi:\mathcal{M}_k\to \Z$$ that sends the class of a $k$-variety
$X$ to the $\ell$-adic Euler characteristic
$$\chi(X)=\sum_{i\geq 0}(-1)^i\mathrm{dim}\,H^i_c(X,\Q_\ell).$$ Since $\chi$ sends $\LL$ to $1$, the
image of $Z_A(T)$ under the morphism $\mathcal{M}_k[[T]]\to
\Z[[T]]$ induced by $\chi$ is equal to
$$\chi(Z_A(T))=\sum_{d\in \N'}\chi(\mathcal{A}(d)_s)T^d.$$

\begin{theorem}
Let $A$ be a tamely ramified abelian $K$-variety. For every $d$ in
$\N'$, we have
$$\sum_{i\geq 0}(-1)^i \mathrm{Trace}(\sigma^d\,|\,H^i(A\times_K
K^t,\Q_\ell))=\chi(\mathcal{A}(d)_s).$$ This value equals
$\phi_A(d)$ if $A(d)$ has purely additive reduction (i.e, if
$\mathcal{A}(d)_s^o$ is unipotent) and it equals zero in all other
cases.
\end{theorem}

This result can be seen as a particular case of a more general
theory that expresses a certain motivic measure for the number of
rational points on a $K$-variety $X$ in terms of the Galois action
on the $\ell$-adic cohomology of $X$; see
\cite{Ni-trace,Ni-tracevar, Ni-saito}. For a similar formula for
the zeta function of a hypersurface singularity, see
\cite[1.1]{DL-Lefschetz}, \cite[9.12]{NiSe-Inv} and
\cite[9.9]{Ni-trace}.

\section{Degenerations of Calabi-Yau varieties}\label{sec-degcy}
We keep the notations from Section \ref{subsec-not}. To simplify
the presentation, we assume that $k$ has characteristic zero. Part
of the theory below can be developed also in the case where $k$
has positive characteristic; see \cite{HaNi-zetav3}. In
particular, the definition of the zeta function remains valid.

\subsection{Motivic zeta functions of Calabi-Yau
varieties}\label{subsec-motCY}

\begin{definition}\label{def-cy}
A Calabi-Yau variety over a field $F$ is a smooth, proper,
geometrically connected $F$-variety with trivial canonical sheaf.
\end{definition}
For instance, every abelian variety is Calabi-Yau. By definition,
every Calabi-Yau variety admits a gauge form. In the definition of
a Calabi-Yau variety $X$, one often includes the additional
condition that $h^{i,0}(X)$ vanishes for $0<i<\mathrm{dim}\,X$. We
do not impose this condition.

Let $X$ be a Calabi-Yau variety over $K$. We will now define the
motivic zeta function $Z_X(T)$ of $X$, in a way that generalizes
our construction for abelian varieties. There is no canonical weak
N\'eron model as in the abelian case, but we can generalize the
expression for the zeta function in Proposition \ref{prop-mothaar}
in terms of the
 motivic volume of an appropriate gauge form on $X$.
We first show how the notion of distinguished gauge form extends
to Calabi-Yau varieties.

\begin{prop}\label{prop-mu}
Let $X$ be a Calabi-Yau variety over $K$, and let $\omega$ be a
gauge form on $X$. Then for every weak N\'eron model $\X$ of $X$,
the value
\begin{eqnarray*}
\ord(X,\omega)&:=&\mathrm{min}\,\{ord_C(\omega)\,|\,C\in
\pi_0(\X_s)\}\quad \in \Z\cup \{-\infty\}
\end{eqnarray*}
only depends on the pair $(X,\omega)$, and not on $\X$. By
convention, we set $\min \emptyset=-\infty$.
\end{prop}
\begin{proof}
Since every connected component of $\X_s$ has the same dimension
as $X$, the value $\ord(X,\omega)$ is precisely minus the {\em
virtual dimension} of the motivic integral
$$\int_{X}|\omega|=\LL^{-\mathrm{dim}(X)}\sum_{C\in
\pi_0(\mathcal{X}_s)}[C]\LL^{-\ord_C\omega}\in \mathcal{M}_k.$$
The virtual dimension of an element $\alpha$ in $\mathcal{M}_k$
can be defined, for instance, as half of the degree of the {\em
Poincar\'e polynomial} of $\alpha$ \cite[\S8]{Ni-tracevar}.
\end{proof}

 Note that
$\ord(X,\omega)=-\infty$ if and only if $\X_s$ is empty, i.e., if
and only if $X(K)$ is empty.

\begin{definition}\label{def-disform}
Let $X$ be a Calabi-Yau variety over $K$. A distinguished gauge
form on $X$ is a gauge form $\omega$ such that $\ord(X,\omega)=0$.
\end{definition}
Thus, a distinguished gauge form on $X$ extends to a relative
differential form on every weak N\'eron model, in a ``minimal''
way. It is clear that
 $X$ admits a distinguished gauge form iff $X$ has a $K$-rational
point, and that a distinguished gauge form is unique up to
multiplication with a unit in $R$.


\begin{definition}\label{def-cyzeta}
Let $X$ be a Calabi-Yau variety over $K$, and assume that $X$ has
a $K$-rational point. Let $\omega$ be a distinguished gauge form
on $X$. We define the motivic zeta function $Z_X(T)$ of $X$ by
$$Z_X(T)=\LL^{\mathrm{dim}(X)} \sum_{d\in \N}\left(\int_{X(d)}|\omega(d)|\right)
T^d\quad \in \mathcal{M}_k[[T]]. $$
\end{definition}
This definition only depends on $X$, and not on the choice of
distinguished gauge form $\omega$, since multiplying $\omega$ with
a unit in $R$ does not affect the motivic integral of $\omega$ on
$X$. It follows from Proposition \ref{prop-mothaar} that, when $X$
is an abelian variety, Definition \ref{def-cyzeta} is equivalent
to  Definition \ref{def-abzeta}.


By embedded resolution of singularities, we can find an
$sncd$-model $\mathcal{X}$ for $X$, i.e., a regular proper
$R$-model such that $\mathcal{X}_s=\sum_{i\in I}N_i E_i$ is a
divisor with strict normal crossings. For every $i\in I$, we
define the order $\mu_i=\ord_{E_i}\omega$ of $\omega$ along $E_i$
as in \cite[6.8]{NiSe-Inv}. These values do not depend on the
choice of distinguished gauge form $\omega$. For every non-empty
subset $J$ of $I$,
we set \begin{eqnarray*}E_J&=&\cap_{j\in J}E_j,\\
E_J^o&=&E_J\setminus (\cup_{i\in I\setminus J}E_i).\end{eqnarray*}
These are locally closed subsets of $\mathcal{X}_s$, and we endow
them with the induced reduced structure. As $J$ runs through the
set of non-empty subsets of $I$, the subvarieties $E_J^o$ form a
partition of $\mathcal{X}_s$.

It follows from \cite[7.7]{NiSe-Inv} that the motivic zeta
function $Z_X(T)$ can be expressed in the form
\begin{equation}\label{eq-ratpres}
Z_X(T)=\sum_{\emptyset\neq J\subset
I}(\LL-1)^{|J|-1}[\widetilde{E}_J^o]\prod_{j \in
J}\frac{\LL^{-\mu_j}T^{N_j}}{1-\LL^{-\mu_j}T^{N_j}}\quad \in
\mathcal{M}_k[[T]]
\end{equation}
where $\widetilde{E}_J^o$ is a certain finite \'etale cover of
$E_J$. By \cite[2.2.2]{Ni-saito}, one can construct
$\widetilde{E}_J^o$ as follows: set
$$N_J=\mathrm{gcd}\{N_j\,|\,j\in J\},$$
 choose a uniformizer $\pi$ in $R$, and denote by $\mathcal{Y}$
 the normalization of
$$\mathcal{X}\times_R (R[x]/(x^{N_J}-\pi)).$$ Then there is an
isomorphism of $E_J^o$-schemes
$$\widetilde{E}_J^o\cong E_J^o\times_{\mathcal{X}}\mathcal{Y}.$$

In particular, one sees from \eqref{eq-ratpres} that $Z_X(T)$ is a
rational function and that every pole of $Z_X(\LL^{-s})$ is of the
form $ s = - \mu_i/N_i $ for some $ i \in I $. Every irreducible
component $E_i$ of the special fiber yields in this way a
``candidate pole'' $-\mu_i/N_i$ of the zeta function. Since the
expression in \eqref{eq-ratpres} is independent of the chosen
normal crossings model $\mathcal{X}$, one expects in general that
not all of these candidate poles are actual poles of $Z_X(T)$.
 But even candidate poles that appear in {\em every} model
 will not always be actual poles. To explain this phenomenon, we
 will propose in Section \ref{subsec-GMP} a version of the
 Monodromy Conjecture for Calabi-Yau varieties.

\begin{example} If $X$ is an elliptic curve, then $X$ admits a unique {\em minimal} regular model with strict
normal crossings $\mathcal{X}$. It is not the case that all
irreducible components of $\mathcal{X}_s$ give actual poles of the
motivic zeta function. This can be seen by combining Theorem
\ref{thm-main} with the Kodaira-N\'eron classification.
\end{example}

\subsection{Log canonical threshold}\label{subsec-lct}
 Let $X$ be a
Calabi-Yau $K$-variety such that $X(K) \neq \emptyset$. From the
formula in \eqref{eq-ratpres}, we see that the poles of
$Z_X(\mathbb{L}^{-s})$ form a finite subset of $\mathbb{Q}$. It
turns out that the \emph{largest} pole of $Z_X(T)$ is an
interesting invariant for $X$, which can be read off from the
numerical data associated to any $sncd$-model of $X$.

 Choose a regular proper
$R$-model $\mathcal{X}$ of $X$ such that $\mathcal{X}_s$ is a
strict normal crossings divisor $\mathcal{X}_s=\sum_{i\in I}N_i
E_i$ and define the values $\mu_i,\,i\in I$ as in Section
\ref{subsec-motCY}. We put
\begin{eqnarray*}
lct(X)&=&\min\{\mu_i/N_i\,|\,i\in I\},
\\ \delta(X)&=&\max \{\,|J|\ |\,\emptyset \neq J\subset I,\
E_J\neq \emptyset,\ \mu_j/N_j=lct(X)\mbox{ for all }j\in J\} - 1.
\end{eqnarray*}

\begin{definition}\label{def-lct}
We call $lct(X)$ the log canonical threshold of $X$, and
$\delta(X)$ the degeneracy index of $X$.
\end{definition}
The following theorem shows that these values do not depend on the
chosen model $\mathcal{X}$.

\begin{theorem}\label{thm-lct}
Let $X$ be a Calabi-Yau variety with $X(K) \neq \emptyset$.
\begin{enumerate}

\item The value $s = -lct(X)$ is the largest pole of the motivic
zeta function $Z_X(\LL^{-s})$, and its order equals $\delta(X) +
1$. In particular, $lct(X)$ and $\delta(X)$ are independent of the
model $\mathcal{X}$. For every integer $d>0$, we have
\begin{eqnarray*}\delta(X\times_K K(d))&=&\delta(X),
\\ lct(X\times_K K(d))&=&d\cdot lct(X).
\end{eqnarray*}

\item Assume moreover that $K=\C((t))$ and that $X$ admits a
projective model $\mathcal{Y}$ over the ring $\C\{t\}$ of germs of
analytic functions at the origin of the complex plane. If we put
$\alpha=lct(X)$, then $\exp(-2\pi i\alpha)$ is an eigenvalue of
the action of the semi-simple part of monodromy on
$$\mathrm{Gr}_F^m H^{m}(\mathcal{Y}_\infty,\C):=\mathrm{Gr}_F^m \mathbb{H}^m(\mathcal{Y}_s^{\an},R\psi_{\mathcal{Y}}(\C))$$
where $ m = \mathrm{dim}(X)$. In particular, for every embedding
of $\Q_\ell$ in $\C$, $\exp(-2\pi i\alpha)$ is an eigenvalue of
every tame monodromy operator $\sigma$ on $H^m(X\times_K
K^t,\Q_\ell)$.
\end{enumerate}
\end{theorem}
In part $(2)$ of Theorem \ref{thm-lct} above,
$H^{m}(\mathcal{Y}_\infty,\C)$ denotes the limit cohomology at
$t=0$ associated to any projective model for $\mathcal{Y}$ over a
small open disc around the origin of $\C$. It carries a natural
mixed Hodge structure \cite{steenbrink,navarro}, and $F^{\bullet}$
denotes the Hodge filtration. Note that $K^t=K^s$ since $k$ has
characteristic zero.

Comparing Theorem \ref{thm-main} and Theorem \ref{thm-lct}, we
find:
\begin{cor}\label{cor-lctab}
If $A$ is an abelian $K$-variety, then $lct(A)=-c(A)$ and
$\delta(A)=t_{\mathrm{pot}}(A)$.
\end{cor}

The degeneracy index of a Calabi-Yau variety $X$ over $K$ is a
measure for the potential degree of degeneration of $X$ over the
closed point of $\Spec R$. If $A$ is an abelian variety, then by
Corollary \ref{cor-lctab}, the degeneracy index $\delta(A)$ is
zero if and only if $A$ has potential good reduction, and
$\delta(A)$ reaches its maximal value $\mathrm{dim}(A)$ if and
only if $A$ has potential purely multiplicative reduction.

Looking at the expression for the zeta function of an abelian
variety in Proposition \ref{prop-refine}, one sees that the zeta
function of an abelian variety encodes many other interesting
invariants of the abelian variety, such as the order function
$\ord_A$ and the number of components $\phi_A(d)$ for every $d$ in
$\N$. Our motivic zeta function allows to generalize these
invariants to Calabi-Yau varieties. Using the expression
\eqref{eq-ratpres}
 for the zeta function in terms of an $sncd$-model, all these invariants
can be explicitly computed on such a model. See
\cite[\S5]{HaNi-zetav3}.
\subsection{Comparison with the case of a hypersurface
singularity}\label{subsec-lctcomparison} Let us return for a moment to the set-up of Section
\ref{subsec-milnor}, still assuming that $k$ has characteristic
zero. We can also apply Definition \ref{def-lct} to this
situation, replacing $X$ by the analytic Milnor fiber
$\mathscr{F}_x$ of $f$ at $x$ and taking for $\omega$ the gauge
form $t \cdot \phi/df$ on $\mathscr{F}_x$, where $\phi/df$ is a
Gelfand-Leray form. In this way, we define the log-canonical
threshold $lct_x(f)$ of $f$ at $x$ and the degeneracy index
$\delta_x(f)$ of $f$ at $x$. One can deduce from
\cite[7.30]{Ni-trace} that $lct_x(f)$ coincides with the usual
log-canonical threshold of $f$ at $x$ as it is defined in
birational geometry. The results in Theorem \ref{thm-lct} remain
valid; in particular, using Theorem \ref{thm-compar}, we see that
$s=-lct_x(f)$ is the largest pole of the motivic zeta function
$Z_{f,x}(\LL^{-s})$ of $f$ at $x$. We refer to \cite{HaNi-zetav3}
 for details.

\subsection{Global Monodromy Property}\label{subsec-GMP}
In the light of our results for abelian varieties, it is natural
to wonder if there is a relation between poles of $Z_X(T)$ and
monodromy eigenvalues for Calabi-Yau varieties $X$, similar to the
one predicted by the motivic monodromy conjecture for hypersurface
singularities (Conjecture \ref{conj-motmon2}).

\begin{definition}\label{def-GMP}
Let $X$ be a Calabi-Yau variety with $X(K) \neq \emptyset$, and
let $\sigma$ be a topological generator of $G(K^t/K)=G(K^s/K)$. We
say that $X$ satisfies the Global Monodromy Property (GMP) if
there exists a finite subset $\mathcal{S} $ of $ \mathbb{Z} \times
\mathbb{Z}_{>0} $ such that
$$ Z_X(T) \in \mathcal{M}_k \left[ T, \frac{1}{1 - \mathbb{L}^a T^b}\right]_{(a,b) \in \mathcal{S}} $$
and such that for each $(a,b) \in \mathcal{S}$, the cyclotomic
polynomial $\Phi_{\tau(a/b)}(t)$ divides the characteristic
polynomial of the monodromy operator $\sigma$ on $H^i(X \times_K
K^t, \mathbb{Q}_{\ell}) $ for some $i \in \mathbb{N}$.
\end{definition}
 Recall that $\tau(a/b)$ denotes the order of $a/b$ in $\Q/\Z$.
By Theorem \ref{thm-main}, every abelian $K$-variety satisfies the
Global Monodromy Property.

\begin{question}\label{ques-GMP}
Is there a natural condition on $X$ that guarantees that $X$
satisfies the Global Monodromy Property (GMP)?
\end{question}

We don't know any example of a Calabi-Yau variety over $K$ that
does not satisfy the GMP. We would like to mention some work in
progress where we can show that the GMP holds for certain types of
varieties ``beyond'' abelian varieties.

\subsubsection*{Semi-abelian varieties}
As a direct generalization of abelian varieties, it is natural to
consider \emph{semi-abelian} varieties, i.e., algebraic $K$-groups
that are extensions of abelian varieties by tori. N\'eron models
exist also for semi-abelian varieties, we refer to \cite{neron}
and \cite{HaNi-zeta} for more details (the N\'eron model we
consider is the maximal quasi-compact open subgroup scheme of the
N\'eron $lft$-model from \cite{neron}). We have generalized
Theorem \ref{thm-main} to tamely ramified semi-abelian
$K$-varieties, in arbitrary characteristic. The main complication
is that one has to control the behaviour of the {\em torsion} part
of the component group of the N\'eron $lft$-model under ramified
base change.

\subsubsection*{$K3$-surfaces.}
Let $X$ be a Calabi-Yau variety over $K$ that admits a
$K$-rational point. To show that $X$ satisfies the Global
Monodromy Property, one strategy would be to consider a regular
proper model $\mathcal{X}$ whose special fiber has strict normal
crossings. In principle, using the expression in
\eqref{eq-ratpres}, one can then determine the poles of $Z_X(T)$.
The next step is to use A'Campo's formula (in the form of
\cite{Ni-saito}) to compute the monodromy zeta function of $X$ on
the model $\mathcal{X}$ (the monodromy zeta function is the
alternate product of the characteristic polynomials of the
monodromy action on the cohomology spaces of $X$).
 In this way, one tries to show that the poles of $Z_X(T)$ correspond
to monodromy eigenvalues. In practice, this kind of argumentation
can be quite complicated. For one thing, when the dimension of $X$
is greater than one, there is usually no distinguished
$sncd$-model to work with, like the minimal $sncd$-model in the
case of elliptic curves. And even when one has some more or less
explicitly given
 model, the combinatorial and geometric complexity of
the special fiber often make computations very hard: one needs to
analyze the model in a very precise way to eliminate fake
candidate poles and to find a sufficiently large list of monodromy
eigenvalues. Worse, the monodromy zeta function might contain too
little information to find all the necessary monodromy
eigenvalues, due to cancellations in the alternate product.

There do however exist cases where this procedure leads to
results. For instance, assume that $X$ has dimension two and that
it allows a triple-point-free degeneration. By this we mean that
$X$ has a proper regular model $\mathcal{X}/R$ where the special
fiber $\mathcal{X}_s$ is a strict normal crossings divisor  such
that three distinct irreducible components of $\mathcal{X}_s$
never meet in one point. Such degenerate triple-point-free fibers
have been classified by B. Crauder and D. Morrison \cite{CrMo}. In
an ongoing project we use their classification to study the
motivic zeta function of $X$, and we have been able to verify in
almost all cases that the Global Monodromy Property holds.

\end{document}